\documentclass{amsart}

\usepackage{amssymb,amsfonts,amsthm,amsmath}
\usepackage{epsfig}
\usepackage[utf8]{inputenc}
\usepackage{mathrsfs}
\usepackage{lscape}
\usepackage{lineno,hyperref}

\newtheorem{theorem}{Theorem}[section]

\theoremstyle{definition}
\newtheorem{definition}[theorem]{Definition}

\newtheorem{teorema}{Theorem}
\newtheorem{lema}{Lemma}

\theoremstyle{remark}
\newtheorem{remark}[theorem]{Remark}

\numberwithin{equation}{section}

\begin{document}

\title{On the Ulam--Hyers stabilities of the solutions of $\Psi$-Hilfer fractional differential equation with abstract Volterra operator}

%    Information for first author
\author{$^{1}$ J. Vanterler da C. Sousa}
%    Address of record for the research reported here
\address{$^{1}$ Mathematics faculty, Federal University of Pará,\\ Augusto Corrêa Street, Number 01, 66075-110, Belém, PA, Brazil}
%    Current address
\email{ra160908@ime.unicamp.br}
%    \thanks will become a 1st page footnote.
%\thanks{The first author was supported in part by NSF Grant \#000000.}

%    Information for second author
\author{$^{2}$ Kishor D. Kucche}
\address{$^{2}$Department of Mathematics, Shivaji University,\\ Kolhapur 416 004, Maharashtra, India}
\email{kdkucche@gmail.com}
%\thanks{Support information for the second author.}

%    Information for terceiro author
\author{$^{3}$ E. Capelas de Oliveira}
\address{$^{3}$ Department of Applied Mathematics, Imecc-Unicamp,\\ 13083-859, Campinas, SP, Brazil}
\email{capelas@ime.unicamp.br}
%\thanks{Support information for the second author.}

%    General info

\date{January 1, 2001 and, in revised form, June 22, 2001.}

%\dedicatory{This paper is dedicated to our advisors.}


\begin{abstract}
In this paper, we consider the new class of the fractional differential equation  involving the abstract Volterra operator in the Banach space and investigate existence, uniqueness and  stabilities of Ulam--Hyers on the compact interval $\Delta=[a,b]$ and on the infinite interval $I=[a,\infty )$. Our analysis is based on the application of the Banach fixed point theorem and the Gronwall inequality involving generalized $\Psi$-fractional integral. At last, we performed out an application to elucidate the outcomes got.
\end{abstract}

\maketitle

%\section*{This is an unnumbered first-level section head}
%This is an example of an unnumbered first-level heading.

%% The correct journal style for \specialsection is all uppercase; a known bug
%% in amsart.cls prevents this, so input must be uppercase until it is fixed.
%\specialsection*{This is a Special Section Head}
%\specialsection*{THIS IS A SPECIAL SECTION HEAD}
%This is an example of a special section head%
%%%%%%%%%%%%%%%%%%%%%%%%%%%%%%%%%%%%%%%%%%%%%%%%%%%%%%%%%%%%%%%%%%%%%%%%
%\footnote{Here is an example of a footnote. Notice that this footnote
%text is running on so that it can stand as an example of how a footnote
%with separate paragraphs should be written.
%\par
%And here is the beginning of the second paragraph.}%
%%%%%%%%%%%%%%%%%%%%%%%%%%%%%%%%%%%%%%%%%%%%%%%%%%%%%%%%%%%%%%%%%%%%%%%%

\section{Introduction}

The study of abstract Volterra equations has long been investigated by a number of researchers such as Kai-Jaung Pei \cite{KaiJaung} and Corduneanu \cite{cordun,cordun1,Corduneanu}. In 2004, Vath \cite{Vath} introduced more general versions of the linear and nonlinear abstract Volterra operator and investigated the existence of local solutions and solution extensions. In 2000, Bedivan and O'Regan \cite{Oregan} investigated the topological structure of the set of solutions for abstract Volterra equations in the Banach and Fr\'echet spaces. In 2010, Serban, Rus and Petrusel \cite{petrusel} investigated the existence, data dependence and comparison results for the solution, using the technique of the operator of Picard. With the progression of time, studies related to the Volterra abstract operator gained interest from a number of consultants, and with this, numerous applications have emerged, such as: control theory, continuous mechanics, nuclear reactor dynamics, linear viscoelastic, among others \cite{Dajun,Laks,kolma}.

On the other hand, the stability investigation of differential and integral equations are important in applications. Hyers, Isac and Rassias \cite{Isac} investigated the stability of functional equations in several variables. Jung \cite{Soon} carried out a work involving fixed point for stability of integral equations of Volterra. Two interesting and important works were investigated by Rassias \cite{Rassias} and Rus \cite{Ioan}, in which they dealt with the stability of a linear mappings in the Banach space and the Ulam stability of ordinary differential equations. It is also important to highlight the work carried out by Otrocol \cite{principal} on Ulam stability of nonlinear differential equations with the abstract Volterra operator in the Banach spaces.

In 2012, Wei, Li and Li \cite{LiLi} investigated new Ulam--Hyers stabilities of Volterra integral equations by means of the fixed-point theorem. Kostic \cite{kostic} has done a work on abstract Volterra integro-differential equations by means of fractional derivative and integral. Other important and interesting works on existence, uniqueness and Ulam--Hyers stabilities, we suggest the references \cite{JERO,exis11}.

The fractional calculus over the time has been gaining increasing prominence in the scientific community \cite{RHM,IP,SAMKO,KSTJ}. With  new definitions  of fractional derivatives and integrals \cite{ZE1,ZE2}, there has arisen some applications in a few territories of learning \cite{apli,apli1,apli2,apli3,apli4,apli6}. There has gained prominence interest in the investigation of existence and Ulam--Hyers stabilities of differential equations and fractional integrals, be they of the impulsive type, functional or of Volterra \cite{JERO,exis8,exis10,imp1,imp2}. Recently, Sousa et al. \cite{est1,est2,est3,est4,est5,est6,est7}, investigated  the existence, uniqueness and stabilities of Ulam--Hyers of fractional differential equations. 

In this sense, by means of the $\Psi$-Hilfer fractional derivative, we introduced a class of nonlinear fractional differential equation with the abstract Volterra operator, in order to investigate a new class of Ulam--Hyers stabilities, contributing with new results and consequently for the growth of the area.

Consider the fractional differential equation with abstract Volterra operator in a Banach space given by 
\begin{equation}\label{sd}
{}^{\mathbf{H}}{\mathbf{D}}_{a^{+}}^{\mu ,\eta ,\Psi }x(t)=f(t,x(t),\widetilde{\mathfrak{W}}%
\left( x\right) (t)]),\qquad t\in \Delta\subset \mathbb{R}
\end{equation}%
where ${}^{\mathbf{H}}{\mathbf{D}}_{a^{+}}^{\mu ,\eta ,\Psi }(\cdot )$ is the $\Psi $-Hilfer fractional derivative of order $0<\mu \leq 1$ and type $
0\leq \eta \leq 1$, $\Delta=[a,b]$ or  $I=[a,\infty )$, $\widetilde{\mathfrak{W}}\in  C((C_{1-\xi;\Psi}[a,b],\mathbb{R}),(C_{1-\xi;\Psi}[a,b],\mathbb{R}))$,  $f\in (I\times 
\mathbb{R}^{2},\mathbb{R})$, where $(\mathbb{R},|\cdot |)$ is a Banach space.

Proposing new results on existence, uniqueness, especially Ulam--Hyers stabilities, is indeed very important and contributes significantly to mathematics, especially for fractional calculus. The fundamental inspiration for  the elaboration of the present work  for Eq.(\ref{sd}) involving abstract Volterra operator originated in order to investigate and propose  existence, uniqueness  and stability results of Ulam--Hyers and Ulam--Hyers--Rassias.

This paper is organized as follows. In section 2 we present the space of the weighted functions and their respective norm. In addition, we present the concepts of $\Psi$-Riemann-Liouville fractional integral, $\Psi$-Hilfer fractional derivative,  the fundamental results in this regard and the Gronwall theorems. The stability concepts of Ulam--Hyers, Ulam--Hyers--Rassias and some important observations conclude the section. In section 3, we investigated the existence, uniqueness and Ulam--Hyers stability of the fractional differential equation Eq.(\ref{sd}) on the compact interval $\Delta=[a,b]$. Section 4 is intended to investigate the uniqueness and Ulam-Hyers-Rassias stability on the infinite interval $I=[a,\infty)$. In section 5, as an application of the results we obtained, we will discuss existence and Ulam--Hyers stability of solution for  fractional Cauchy problem involving the Hadamard derivative. The concluding remarks close the paper.

%%%%%%%%%%%%%%%%%%%%%%%%%%%%%%%%%%%%%%%%%%%%%%%%%%%%%%%%%%%%%%%%%%%%%%%%%%%%%%%%
%%%%%%%%%%%%%%%%%%%%%%%%%%%%%%%%%%%%%%%%%%%%%%%%%%%%%%%%%%%%%%%%%%%%%%%%%%%%%%%%
\section{Preliminaries}
In this section, we introduce the space of weighted functions and their respective norm. We introduce the concepts of $\Psi$-Riemann-Liouville fractional integral, $\Psi$-Hilfer fractional derivative and the fixed-point theorem of Banach. In addition, we present the generalized Gronwall theorems to investigate the stabilities of Ulam--Hyers and Ulam--Hyers--Rassias. Some important remarks close the section.

Let $\Delta=[a,b]$ or  $I=[a,\infty )$ and $C(\Delta,\mathbb{R})$ the space of continuous functions with norm \cite{ZE1,est1}
\begin{equation}
\left\Vert x\right\Vert =\underset{t\in \Delta}{\sup }\left\vert x\left( t\right)
\right\vert .
\end{equation}

The weighted space $C_{1-\xi ,\Psi }(I,\mathbb{R})$ of functions $x$ is defined by \cite{ZE1,est1}
\begin{equation}
C_{1-\xi ,\Psi }(\Delta,\mathbb{R})=\left\{ x\in C(I_{2},\mathbb{R}),(\Psi
(t)-\Psi (a))^{1-\xi }x(t)\in C(\Delta,\mathbb{R})\right\}, 
\end{equation}%
where $0\leq \xi \leq 1$ with norm 
\begin{equation}
\left\Vert x\right\Vert _{C_{1-\xi ,\Psi }}=\underset{t\in \Delta}{\sup }%
\left\vert (\Psi (t)-\Psi (a))^{1-\xi }x(t)\right\vert .
\end{equation}%

Obviously, the space $C_{1-\xi; \Psi }(\Delta,\mathbb{R})$ is a Banach space.
The weighted space $C_{1-\xi; \Psi }^{n}(\Delta,\mathbb{R})$ of functions $x$ is defined by  
\begin{equation}
C_{1-\xi ,\Psi }^{n}(\Delta,\mathbb{R})=\{x:(a,b]\rightarrow \mathbb{R};\text{ 
}x(t)\in C^{n-1}(\Delta,\mathbb{R};x^{(n)}(t)\in C_{1-\xi ,\Psi }(\Delta,\mathbb{R}%
)\},
\end{equation}%
where $0\leq \xi \leq 1$, with norm 
\begin{equation}
\left\Vert x\right\Vert _{C_{1-\xi ,\Psi }^{n}(\Delta,\mathbb{R}%
)}=\sum_{k=0}^{n-1}\left\Vert x^{(k)}\right\Vert _{C(\Delta,\mathbb{R}%
)}+\left\Vert x^{(n)}\right\Vert _{C_{1-\xi ,\Psi }(\Delta,\mathbb{R})}.
\end{equation}

Let $I_{2}=(a,b)$ $(-\infty \leq a<b\leq \infty )$ be a finite or infinite
interval of the real line $\mathbb{R}$ and let $\mu >0$. Also, let $\Psi
(t)$ be an increasing and positive monotone function on $I_{3}=(a,b]$,
having a continuous derivative $\Psi ^{\prime }(t)$ on $I_{1}$. The
left-sided fractional integral of a function $f$ with respect to function $%
\Psi $ on $[a,b]$ is defined by \cite{ZE1}
\begin{equation}
\mathbf{I}_{a^{+}}^{\mu ,\Psi }x(t)=\dfrac{1}{\Gamma (\mu )}\int_{a}^{t}\mathbf{Q}_{\Psi
}^{\mu }\left( t,s\right) x(s)\,{\mbox{d}}s,
\end{equation}
where $\mathbf{Q}_{\Psi }^{\mu }\left( t,s\right) :=\Psi ^{\prime }\left( s\right)\left( \Psi \left( t\right) -\Psi \left( s\right) \right) ^{\mu -1}.$ The right-sided fractional integral is defined in an analogous form.

On the other hand, let $n-1 < \mu \leq n$ with $n \in \mathbb{N}$, $%
\Delta=[a,b]$ an interval such that $(-\infty \leq a < b \leq \infty)$ and let $%
f, \Psi \in C^n(\Delta,\mathbb{R})$ be two functions such that $\Psi$ is
increasing and $\Psi^{\prime }(t) \neq 0$, for all $t \in \Delta$. The left-sided 
$\Psi$-Hilfer fractional derivative ${}^{\mathbf{H}} {\mathbf{D}}_{a^{+}}^{\mu,%
\eta,\Psi}(\cdot)$ of a function $f$, of order $\mu$ $(0 < \mu \leq 1)
$ and type $\eta$ $(0 \leq \eta \leq 1)$, is defined by \cite{ZE1}
\begin{equation*}
{}^{\mathbf{H}} {\mathbf{D}}_{a^{+}}^{\mu,\eta,\Psi} x(t)=
\mathbf{I}_{a^{+}}^{\eta(n-\mu),\Psi} \left(\frac{1}{\Psi^{\prime }(t)} \frac{{%
\mbox{d}}}{{\mbox{d}}t} \right)^n \mathbf{I}_{a{+}}^{(1-\eta)(n-\mu),\Psi}
x(t). 
\end{equation*}

The right-sided $\Psi$-Hilfer fractional derivative is defined in an
analogous form.

\begin{teorema}{\rm \cite{ZE1}}
\label{T2} If $x\in C_{1-\xi ,\Psi }^{1}(\Delta,\mathbb{R}),\,\,0<\mu \leq 1
$ and $0\leq \eta \leq 1$, then 
\begin{equation*}
\mathbf{I}_{a^{+}}^{\mu ,\Psi }{}^{\mathbf{H}}{\mathbf{D}}_{a^{+}}^{\mu ,\eta ,\Psi
}x(t)=x(t)-\mathbf{M}^{\Psi}_{\xi}(t,a)\mathbf{I}_{a^{+}}^{(1-\eta )(1-\mu ),\Psi }x(a),
\end{equation*}
where $\mathbf{M}^{\Psi}_{\xi}(t,a):= \dfrac{(\Psi (t)-\Psi (a))^{\xi -1}}{\Gamma (\xi )}$.
\end{teorema}

\begin{teorema}{\rm \cite{ZE1}}
\label{T3} Let $x\in C_{1-\xi ,\Psi }^{1}(\Delta,\mathbb{R}),\mu >0$ and $%
0\leq \eta \leq 1$, then we have 
\begin{equation*}
{\mathbf{D}}_{a^{+}}^{\mu ,\eta ,\Psi }\mathbf{I}_{a^{+}}^{(1-\eta )(1-\mu
),\Psi }x(t)=x(t).
\end{equation*}
\end{teorema}

\begin{teorema}\label{wul}{\rm \cite{est4}} Let $\left( X,d\right) $ be a generalized complete metric space. Assume that $\Omega :X\rightarrow X$ is a strictly contractive operator with the Lipschitz constant $L<1.$ If there exists a nonnegative integer $k$ such that $d\left( \Omega ^{k+1},\Omega ^{k}\right) <\infty $ for some \ $x\in X\, $, then the following are true:
\begin{enumerate}
\item The sequence $\left\{ \Omega ^{k}x\right\} $ converges to a point $x^{\ast }$ of $\Omega$;

\item $x^{\ast }$ is the unique fixed point of $\Omega $ in $\Omega ^{\ast }=\left\{ y\in X/d\left( \Omega ^{k}x,y\right) <\infty \right\}$;

\item If $y\in X^{\ast }$, then $d\left( y,x^{\ast }\right) \leq \dfrac{1}{1-L}d\left( \Omega y,y\right)$.
\end{enumerate}
\end{teorema}

\begin{teorema}{\rm \cite{gron}}
\label{gronwalltheorem}  \textsc{(Gronwall theorem).} Let $u,v$ be two integrable functions and $g$ a continuous function, it domain $[a,b]$. Let $ \Psi \in C^1(\Delta,\mathbb{R})$  an increasing function such that $\Psi^{\prime }(t) \neq 0$, $\forall t \in \Delta$. Assume that:  
\begin{enumerate}
\item $u$ and $v$ are nonnegative;
\item $y$ is nonnegative and nondecreasing.
\end{enumerate}

If 
\begin{equation*}
u(t)\leq v(t)+g(t)\int_{a}^{b}\mathbf{Q}_{\Psi }^{\mu }\left( t,s\right) u(s)\,{%
\mbox{d}}s,
\end{equation*}%
then 
\begin{equation*}
u(t)\leq v(t)+\int_{a}^{b}\sum_{k=1}^{\infty }\frac{[g(t)\xi (\mu
)]^{k}}{\Gamma (\mu k)}\mathbf{Q}_{\Psi }^{k\mu }\left( t,s\right) v(s)\,{%
\mbox{d}}s,
\end{equation*}%
$\forall t\in \Delta$ and $\mathbf{Q}_{\Psi }^{k\mu }\left( t,s\right) :=\Psi ^{\prime }\left( s\right) \left( \Psi \left( t\right) -\Psi \left( s\right) \right) ^{k\mu -1}.$
\end{teorema}

\begin{lema} {\rm\cite{gron}}
\label{L1} \textsc{(Gronwall lemma)} Under the hypotheses of {\rm Theorem \ref{gronwalltheorem}}, let $v$ be a non-decreasing function on $\Delta$. Then, we have 
\begin{equation*}
u(t)\leq v(t)\mathcal{E}_{\mu }\left( g(t)\Gamma (\mu )[(\Psi (t)-\Psi
(a))^{\mu }]\right) 
\end{equation*}%
$t\in \Delta$, where $\mathcal{E}_{\mu }(\cdot )$ is the Mittag-Leffler function with one parameter.
\end{lema}

Let $(\mathbb{R},|\cdot|)$ be a Banach space and $\widetilde{\mathfrak{W}}:(C_{1-\xi;\Psi}\Delta,\mathbb{R}) \to (C_{1-\xi;\Psi}\Delta,\mathbb{R})$
an abstract Volterra operator.

For $f \in C_{1-\xi;\Psi} (\Delta \times \mathbb{R}^2, \mathbb{R})$, $\varepsilon
> 0$ and $\varphi \in C_{1-\xi,\Psi} (\Delta,\mathbb{R}_{+})$ we consider the following fractional Cauchy problem
\begin{eqnarray}
{}^{\mathbf{H}}{\mathbf{D}}_{a^{+}}^{\mu ,\eta ,\Psi }x(t) &=& f(t,x(t),\widetilde{\mathfrak{W}}\left( x\right) (t))  \label{1} \\
\mathbf{I}_{a^{+}}^{1-\xi ,\Psi }x(a) &=&\delta ,~ \delta \in \mathbb{R}
\label{11}
\end{eqnarray}
and the inequalities given below
\begin{equation}
\left\vert {}^{\mathbf{H}}{\mathbf{D}}_{a^{+}}^{\mu ,\eta ,\Psi }y(t)-f(t,y(t),\widetilde{\mathfrak{W}}\left( y\right) (t)\right\vert \leq \varepsilon ,\quad t\in \Delta,
\label{2}
\end{equation}%
\begin{equation}
\left\vert {}^{\mathbf{H}}{\mathbf{D}}_{a^{+}}^{\mu ,\eta ,\Psi }y(t)-f(t,y(t),%
\widetilde{\mathfrak{W}}y(t)\right\vert \leq \varphi (t),\quad t\in \Delta.  \label{3}
\end{equation}

To deal with different kinds of Ulam types stabilities of {\rm Eq.(\ref{1})} we adopt the definitions of \cite{principal}.

\begin{definition}
The {\rm Eq.(\ref{1})} is Ulam--Hyers stable if there exist a real number $c > 0$ such that for each $\varepsilon > 0$ and for each  solution $y \in
C_{1-\xi,\Psi}^1(\Delta,\mathbb{R})$ of inequality {\rm(\ref{2})} there
exist a solution $x \in C_{1-\xi,\Psi}^1(\Delta,\mathbb{R})$ of {\rm Eq.(
\ref{1})} with  
\begin{equation*}
\left\Vert y-x\right\Vert _{C_{1-\xi;\Psi }(\Delta,\mathbb{R})}  \leq c \, \varepsilon.
\end{equation*}
\end{definition}

\begin{definition}
The {\rm Eq.(\ref{1})} is generalized Ulam-Hyers-Rassias stable with respect to $\varphi$, if there exist $c_{\varphi} > 0$ such that for each solution $y \in C_{1-\xi,\Psi}^1(\Delta,\mathbb{R})$ of inequality {\rm(%
\ref{3})} with  
\begin{equation*}
(\Psi (t)-\Psi (a))^{1-\xi }| y(t) - x(t) | \leq c_{\varphi} \varphi(t), \, t \in I.  
\end{equation*}
\end{definition}

\begin{remark}
A function $y \in C_{1-\xi,\Psi}^1 (\Delta,\mathbb{R})$ satisfies inequality {\rm (\ref{2})} if and only if there exists a function $g \in C_{1-\xi,\Psi}(\Delta,\mathbb{R})$ such that
\begin{enumerate}
\item $\left\vert g(t)\right\vert \leq \varepsilon ,\quad t\in \Delta$;

\item ${}^{\mathbf{H}}{\mathbf{D}}_{a^{+}}^{\mu ,\eta ,\Psi }y(t)=f\left( t,y(t),
\widetilde{\mathfrak{W}}\left( y\right) (t)\right) +g(t),\quad t\in \Delta$.
\end{enumerate}

\end{remark}

\begin{remark}
A function $y \in C_{1-\xi,\Psi}^1 (\Delta,\mathbb{R})$ satisfies inequality {\rm (\ref{3})} if and only if there exists a function $\widetilde{g} \in C(\Delta,\mathbb{R})$ {\rm (which depends on $y$)} such that
	
\begin{enumerate}
\item  $\left\vert \widetilde{g}(t)\right\vert \leq \varphi (t),\quad t\in \Delta$;

\item $^{\mathbf{H}}{\mathbf{D}}_{a^{+}}^{\mu ,\eta ,\Psi }y(t)=f\left( t,y(t), \widetilde{\mathfrak{W}}\left( y\right) (t)\right) +\widetilde{g}(t),\quad t\in \Delta$.
\end{enumerate}
\end{remark}

\begin{remark}\label{R5}
If $y \in C_{1-\xi,\Psi}^1 (\Delta,\mathbb{R})$ satisfies inequality {\rm (\ref{2})}, then $y$ is a solution of the following integral equation
\begin{eqnarray*}
&&\left\vert y(t)-y(a)\mathbf{M}^{\Psi}_{\xi}(t,a){\Gamma(\xi )%
}-\frac{1}{\Gamma (\mu )}\int_{a}^{t}\mathbf{Q}_{\Psi }^{\mu }\left( t,s\right)
f(s,y(s),\widetilde{\mathfrak{W}}\left( y\right)(s)\,{\mbox{d}}s\right\vert   \notag \\
&\leq &\frac{(\Psi (t)-\Psi (a))^{\mu }}{\Gamma (\mu +1)}\varepsilon
,\quad t\in \Delta.
\end{eqnarray*}
\end{remark}

\begin{remark} \label{R55}
If $y \in C_{1-\xi,\Psi}^1 (\Delta,\mathbb{R})$ satisfies inequality {\rm (\ref{3})}, then $y$ is a solution of the following integral equation
\begin{eqnarray*}
&&\left\vert y(t)-y(a)\mathbf{M}^{\Psi}_{\xi}(t,a)-\frac{1}{\Gamma (\mu )}\int_{a}^{t}\mathbf{Q}_{\Psi }^{\mu }\left( t,s\right)f(s,y(s),\widetilde{\mathfrak{W}}\left( y\right) (s)\,{\mbox{d}}s\right\vert   \notag \\
&\leq &\frac{1}{\Gamma (\mu )}\int_{a}^{b}\mathbf{Q}_{\Psi }^{\mu }\left( t,s\right) \varphi (s)\,{\mbox{d}}s,\quad t\in \Delta.
\end{eqnarray*}
\end{remark}

%%%%%%%%%%%%%%%%%%%%%%%%%%%%%%%%%%%%%%%%%%%%%%%%%%%%%%%%%%%%%%%%%%%%%%%%%%%%%%%%
%%%%%%%%%%%%%%%%%%%%%%%%%%%%%%%%%%%%%%%%%%%%%%%%%%%%%%%%%%%%%%%%%%%%%%%%%%%%%%%%

\section{Ulam--Hyers stability} 

In this section, our main results investigate the existence, uniqueness and stability of Ulam--Hyers of the Cauchy fractional problem (\ref{1})--(\ref{11}) on the compact interval $\Delta=[a,b]$.

\begin{teorema}\label{T0} Consider the following;

{\rm(a)} $f \in C_{1-\xi,\Psi}(\Delta \times \mathbb{R}^2, \mathbb{R} ), ~\widetilde{\mathfrak{W}} \in C((C_{1-\xi,\Psi}\Delta,\mathbb{R}), (C_{1-\xi,\Psi}\Delta,\mathbb{R}) )$;

{\rm(b)} There exist $L_{f}>0$ such that 
\begin{equation*}
\left\vert f(t,u_{1},u_{2})-f(t,v_{1},v_{2})\right\vert \leq L_{f}\sum_{k=1}^{2}\left\vert u_{i}-v_{i}\right\vert ,~ t\in \lbrack a,b], 
\end{equation*}
with $u_{i},v_{i}\in \mathbb{R}$, $i=1,2$;

{\rm(c)} There exists $L_{\widetilde{\mathfrak{W}}}>0$ such that 
\begin{equation*}
\left\vert \widetilde{\mathfrak{W}}\left( x\right) (t)-\widetilde{\mathfrak{W}}\left( y\right)
(t)\right\vert \leq L_{\widetilde{\mathfrak{W}}}\left\vert x(t)-y(t)\right\vert ,\quad x,\,y\in
C_{1-\xi ,\Psi }[a,b],\,\,t\in \Delta;
\end{equation*}

{\rm (d)} The inequality $L_{f}\left\{ 1+L_{\widetilde{\mathfrak{W}}}\right\} \dfrac{\Gamma(\xi
)(\Psi (b)-\Psi (a))^{\mu }}{\Gamma (\xi +\mu )}<1$ is true.

Then, we have

{\rm (i)} The fractional problem {\rm Eq.(\ref{1})} and {\rm Eq.(\ref{11})} has a unique solution in $C_{1-\xi,\Psi}(\Delta,\mathbb{R})$;

{ \rm (ii)} The solution of {\rm Eq.(\ref{1})} is Ulam--Hyers stable.
\end{teorema}

\begin{proof}
{\rm(i)}. Under condition {\rm(a)}, {\rm Eq.(\ref{1})} and 
{\rm Eq.(\ref{11})} are equivalent to the integral equation 
\begin{equation}
x(t)=\mathbf{M}^{\Psi}_{\xi}(t,a)\delta +\frac{1}{\Gamma (\mu )}\int_{a}^{t}\mathbf{Q}_{\Psi }^{\mu }\left( t,s\right) f\left(s,x(s),\widetilde{\mathfrak{W}}\left( x\right) (s)\right) \,{\mbox{d}}s.
\end{equation}

In fact, applying the integral operator $\mathbf{I}_{a_{+}}^{\mu ,\Psi }(\cdot )$ on both sides of {\rm Eq.(\ref{1})}, using the relation $ \mathbf{I}_{a_{+}}^{1-\xi ,\Psi }x(a)=\delta $ and {\rm Theorem \ref{T2}}, we have 
\begin{equation*}
x(t)-\mathbf{M}^{\Psi}_{\xi}(t,a) \mathbf{I}_{a_{+}}^{(1-\eta )(1-\mu ),\Psi }x(a)=\mathbf{I}_{a{+}}^{\mu ,\Psi }f\left(t,x(t),\widetilde{\mathfrak{W}}\left( y\right) (t)\right) 
\end{equation*}
which implies that 
\begin{equation}
x(t)=\mathbf{M}^{\Psi}_{\xi}(t,a)\delta
+\mathbf{I}_{a_{+}}^{\mu ,\Psi }f\left( t,x(t),\widetilde{\mathfrak{W}}\left( y\right) (t)\right) .  \label{A}
\end{equation}

On the other hand, applying the $\Psi$-Hilfer fractional derivative ${}^{\mathbf{H}}%
\mathbf{D}_{a_{+}}^{\mu ,\eta ,\Psi }(\cdot )$ on both sides of {\rm
Eq.(\ref{A})} and using {\rm Theorem \ref{T3}}, we have 
\begin{eqnarray}
^{\mathbf{H}}\mathbf{D}_{a_{+}}^{\mu ,\eta ,\Psi }x(t) &=&\displaystyle{}^{\mathbf{H}}\mathbf{D}_{a_{+}}^{\mu ,\eta ,\Psi }\left[ \mathbf{M}^{\Psi}_{\xi}(t,a)\delta +\mathbf{I}_{a{+}}^{\mu ,\Psi }f\left( t,x(t),\widetilde{\mathfrak{W}}\left( y\right) (t)\right) \right]   \notag \\
&=&f\left( t,x(t),\widetilde{\mathfrak{W}}\left( y\right) (t)\right) 
\end{eqnarray}
where 
\begin{equation*}
{}^{\mathbf{H}}\mathbf{D}_{a_{+}}^{\mu ,\eta ,\Psi }\left[ \mathbf{M}^{\Psi}_{\xi}(t,a)\delta \right] =0
\end{equation*}%
with $0<\xi <1$.

Now, consider $X=C_{1-\xi ,\Psi }\left( \Delta,\mathbb{R}\right) $ and the operator $\mathscr{B}_{f}:X\rightarrow X$  given by 
\begin{equation*}
\mathscr{B}_{f}x(t)=\mathbf{M}^{\Psi}_{\xi}(t,a)%
\delta +\frac{1}{\Gamma (\mu )}\int_{a}^{t}\mathbf{Q}_{\Psi }^{\mu }\left(
t,s\right) f\left( s,x(s),\widetilde{\mathfrak{W}}\left( y\right) (s)\right) \,{\mbox{d}%
}s.
\end{equation*}
The main purpose here is to prove that $\mathscr{B}_{f}$ is a contraction on $X$ with respect to the norm $\left\Vert \cdot \right\Vert _{C_{1-\xi ,\Psi }}$. 

For any  $x,y\in C_{1-\xi ,\Psi }(\Delta,\mathbb{R})$ and $t\in \Delta=[a,b]$, we have 
\begin{eqnarray*}
&&\left\vert  \mathscr{B}_{f}\left( x(t)\right) -\mathscr{B}_{f}\left(
y(t)\right) \right\vert  \\
&\leq &\frac{1}{\Gamma (\mu )}\int_{a}^{t}\mathbf{Q}_{\Psi }^{\mu
}\left( t,s\right) \left\vert f(s,x(s),\widetilde{\mathfrak{W}}\left(
x\right) (s)-f(s,y(s),\widetilde{\mathfrak{W}}\left( y\right) (s)\right\vert
\,{\mbox{d}}s \\
&\leq &\frac{1}{\Gamma (\mu )}\int_{a}^{t}\mathbf{Q}_{\Psi }^{\mu
}\left( t,s\right) L_{f}\left\{ \left\vert x(s)-y(s)\right\vert +\left\vert 
\widetilde{\mathfrak{W}}\left( x\right) (s)-\widetilde{\mathfrak{W}}\left(
y\right) (s)\right\vert \right\} \,{\mbox{d}}s \\
&\leq &L_{f}\left\{ 1+L_{\widetilde{\mathfrak{W}}}\right\} \frac{1}{\Gamma
(\mu )}\int_{a}^{t}\mathbf{Q}_{\Psi }^{\mu }\left( t,s\right)
\,\left\vert x(s)-y(s)\right\vert {\mbox{d}}s \\
&\leq &L_{f}\left\{ 1+L_{\widetilde{\mathfrak{W}}}\right\} \left\Vert x-y\right\Vert _{C_{1-\xi ;\Psi }(\Delta,\mathbb{R})}\frac{1}{\Gamma (\mu )}\int_{a}^{t}\mathbf{Q}_{\Psi }^{\mu }\left( t,s\right) \left( \Psi \left( s\right)
-\Psi \left( a\right) \right) ^{\xi -1}{\mbox{d}}s \\ &=&L_{f}\left\{ 1+L_{\widetilde{\mathfrak{W}}}\right\} \left\Vert x-y\right\Vert _{C_{1-\xi ;\Psi }(\Delta,\mathbb{R})}\mathbf{I}_{a{+}}^{\mu ;\Psi }\left( \Psi \left( t\right) -\Psi \left( a\right) \right) ^{\xi -1} \\ &=&L_{f}\left\{ 1+L_{\widetilde{\mathfrak{W}}}\right\} \left\Vert
x-y\right\Vert _{C_{1-\xi ;\Psi }(\Delta,\mathbb{R})}\frac{\Gamma \left( \xi \right) \left( \Psi \left( t\right) -\Psi \left( a\right) \right) ^{\mu +\xi -1}}{\Gamma \left( \mu +\xi \right) }
\end{eqnarray*}%

Therefore, we get 
\begin{eqnarray}\label{js}
& &\left\Vert \mathscr{B}_{f}x-\mathscr{B}_{f}y\right\Vert _{C_{1-\xi };\Psi }\nonumber \\
&= & \underset{t\in I}{\sup }%
\left\vert (\Psi (t)-\Psi (a))^{1-\xi }\mathscr{B}_{f}x(t)-\mathscr{B}_{f}y (t)\right\vert\nonumber \\
&\leq & L_{f}\left\{ 1+L_{\widetilde{\mathfrak{W}}}\right\} \left\Vert
x-y\right\Vert _{C_{1-\xi ;\Psi }(\Delta,\mathbb{R})}\frac{\Gamma \left( \xi \right) }{%
\Gamma \left( \mu +\xi \right) }~~\underset{t\in I}{\sup }\left( \Psi
\left( t\right) -\Psi \left( a\right) \right) ^{\mu }  \nonumber \\
&\leq & L_{f}\left\{ 1+L_{\widetilde{\mathfrak{W}}}\right\} \frac{\Gamma
\left( \xi \right) }{\Gamma \left( \mu +\xi \right) }\left( \Psi
\left( b\right) -\Psi \left( a\right) \right) ^{\mu }\left\Vert
x-y\right\Vert _{C_{1-\xi ;\Psi }(\Delta,\mathbb{R})}.
\end{eqnarray}

Condition (d) ensures that $\mathscr{B}_{f}$ is a contraction with the norm $\left\Vert \cdot \right\Vert _{C_{1-\xi ,\Psi }(\Delta,\mathbb{R})}$ on $C_{1-\xi ,\Psi }(\Delta,\mathbb{R})$. Thus, by means of Banach fixed point ({\rm Theorem \ref{wul}}) the problem {\rm Eq.(\ref{1})} and {\rm Eq.(\ref{11})} has a unique solution in $C_{1-\xi ,\Psi }(\Delta,\mathbb{R})$. This concludes the first part of the proof.

{\rm(ii)} Now, we consider the $y\in C_{1-\xi ,\Psi }^{1}\left( \Delta,\mathbb{R}\right) $ satisfying the fractional inequality {\rm Eq.(\ref{2})}  and we denote by $x\in C_{1-\xi ,\Psi }^{1}(\Delta,\mathbb{R})$ the unique solution to the fractional Cauchy problem
\begin{equation}
\left\{ 
\begin{array}{ccl}
{}^{\mathbf{H}}{\mathbf{D}}_{a^{+}}^{\mu ,\eta ;\Psi }x(t) & = & f\left( t,x(t),%
\widetilde{\mathfrak{W}}\left( x\right) (t)\right) ,~ t\in \Delta \\ 
\mathbf{I}_{a{+}}^{1-\xi ;\Psi }x(a) & = & y(a).%
\end{array}%
\right. 
\end{equation}

Using the condition {\rm(a)}, we obtain
\begin{equation}
x(t)=\mathbf{M}^{\Psi}_{\xi}(t,a)y(a)+\dfrac{1}{
\Gamma (\mu )}\int_{a}^{t}\mathbf{Q}_{\Psi }^{\mu }\left( t,s\right) f\left(
x,x(s),\widetilde{\mathfrak{W}}\left( x\right) (s)\right) \,{\mbox{d}}s,
\end{equation}%
for $t\in \Delta$. Since $y\in C_{1-\xi ,\Psi }^{1}\left( \Delta,\mathbb{R}\right) $ satisfies inequality {\rm Eq.(\ref{2})}, by means of the {\rm Remark \ref{R5}}, we have
\begin{eqnarray*}
&&\left\vert y(t)-y(a)\mathbf{M}^{\Psi}_{\xi}(t,a)-\frac{1}{\Gamma (\mu )}\int_{a}^{t}\mathbf{Q}_{\Psi }^{\mu }\left( t,s\right)
f\left( s,y(s),\widetilde{\mathfrak{W}}\left( y\right) (s)\right) \,{\mbox{d}}%
s\right\vert  \\
&\leq &\frac{(\Psi (t)-\Psi (a))^{\mu }}{\Gamma (\mu +1)}\varepsilon.
\end{eqnarray*}%

Therefore for $t\in \lbrack a,b]$, we can write 
\begin{eqnarray}
&&\left\vert y(t)-x(t)\right\vert   \notag \\
&=&\left\vert y(t)-\mathbf{M}^{\Psi}_{\xi}(t,a)
y(a)-\dfrac{1}{\Gamma (\mu )}\int_{a}^{t}\mathbf{Q}_{\Psi }^{\mu }\left(
t,s\right) f\left( s,x(s),\widetilde{\mathfrak{W}}\left( x\right) (s)\right) \,{\mbox{d}%
}s\right\vert   \notag \\
&\leq &\left\vert y(t)-\mathbf{M}^{\Psi}_{\xi}(t,a)y(a)-\dfrac{1}{\Gamma (\mu )}\int_{a}^{t}\mathbf{Q}_{\Psi }^{\mu }\left(
t,s\right) f\left( s,y(s),\widetilde{\mathfrak{W}}\left( y\right) (s)\right) \,{\mbox{d}%
}s\right\vert +  \notag \\
&&+\dfrac{1}{\Gamma (\mu )}\int_{a}^{t}\mathbf{Q}_{\Psi }^{\mu }\left(
t,s\right) \left\vert f\left( s,y(s),\widetilde{\mathfrak{W}}\left( y\right) (s)\right)
-f\left( s,x(s),\widetilde{\mathfrak{W}}\left( x\right) (s)\right) \right\vert \,{%
\mbox{d}}s  \notag \\
&\leq &\frac{(\Psi (t)-\Psi (a))^{\mu }}{\Gamma (\mu +1)}\varepsilon +\dfrac{L_{f}}{\Gamma(\mu
)}\int_{a}^{t}\mathbf{Q}_{\Psi }^{\mu }\left( t,s\right) \left\{
|y(s)-x(s)|+L_{\widetilde{\mathfrak{W}}}\left\vert y(s)-x(s)\right\vert \right\} \,{\mbox{d}}s 
\notag \\
&=&\frac{(\Psi (t)-\Psi (a))^{\mu }}{\Gamma (\mu +1)}\varepsilon +\dfrac{%
L_{f}(1+L_{\widetilde{\mathfrak{W}}})}{\Gamma (\mu )}\int_{a}^{t}\mathbf{Q}_{\Psi }^{\mu }\left(
t,s\right) |y(s)-x(s)|\,{\mbox{d}}s.
\end{eqnarray}

Using the Gronwall lemma {\rm (Lemma \ref{L1})} we get, for $\Delta$, 
\begin{equation}
\left\vert y(t)-x(t)\right\vert \leq \frac{(\Psi (t)-\Psi (a))^{\mu }}{\Gamma (\mu +1)}\varepsilon  \mathcal{E}_{\mu }\left[ \frac{L_{f}(1+L_{\widetilde{\mathfrak{W}}})}{\Gamma (\mu )}\Gamma (\mu )(\Psi (t)-\Psi (a))^{\mu }\right]. 
\end{equation}

Therefore
\begin{align*}
\left\Vert y-x\right\Vert _{C_{1-\xi };\Psi } 
&=  \underset{t\in I}{\sup }%
\left\vert (\Psi (t)-\Psi (a))^{1-\xi }y(t)-x(t) \right\vert \\
&\leq \frac{(\Psi (b)-\Psi (a))^{\mu+1-\xi }}{\Gamma (\mu +1)}  \mathcal{E}_{\mu }\left[ L_{f}(1+L_{\widetilde{\mathfrak{W}}})(\Psi (b)-\Psi (a))^{\mu }\right]\varepsilon\\
&=c\,\varepsilon ,
\end{align*}
where 
\begin{equation*}
c:=\frac{(\Psi (b)-\Psi (a))^{\mu+1-\xi }}{\Gamma (\mu +1)}  \mathcal{E}_{\mu }\left[ L_{f}(1+L_{\widetilde{\mathfrak{W}}})(\Psi (b)-\Psi (a))^{\mu }\right]
\end{equation*}%
with $\mathcal{E}_{\mu }(\cdot )$ a Mittag-Leffler function. This proves the solution of the problem {\rm Eq.(\ref{1})} and {\rm Eq.(\ref{11})} is Ulam--Hyers stable. 
\end{proof}

%%%%%%%%%%%%%%%%%%%%%%%%%%%%%%%%%%%%%%%%%%%%%%%%%%%%%%%%%%%%%%%%%%%%%%%%%%%%%%%%
%%%%%%%%%%%%%%%%%%%%%%%%%%%%%%%%%%%%%%%%%%%%%%%%%%%%%%%%%%%%%%%%%%%%%%%%%%%%%%%%

\section{Generalized Ulam-Hyers-Rassias stability}

In this section, another our main result investigate the uniqueness and stability of Ulam-Hyers-Rassias of the Cauchy fractional problem Eq.(\ref{1})--Eq.(\ref{11}) on the infinite interval $I=[a,\infty )$.

\begin{teorema}\label{T6}
Consider the following:

{\rm($\tilde{a}$)} $f\in C_{1-\xi ,\Psi }([a,\infty )\times \mathbb{R}^{2}, \mathbb{R})$; $\widetilde{\mathfrak{W}}\in C((C_{1-\xi ,\Psi }\Delta,\mathbb{R}),(C_{1-\xi ,\Psi }(\Delta,\mathbb{R})))$;

{\rm($\tilde{b}$)} There exists non-decreasing function $\widetilde{L}_{f}\in C([a,\infty ),\mathbb{R}%
_{+})$ such that 
\begin{equation*}
|f(t,u_{1},u_{2})-f(t,v_{1},v_{2})|\leq \widetilde{L}_{f}(t)\left( \left\vert
u_{1}-v_{1}\right\vert +|u_{2}-v_{2}|\right),
\end{equation*}
for $t\in \lbrack a,\infty )$, and $u_{i},v_{i}\in \mathbb{R}$ with $i=1,2$;

{\rm($\tilde{c}$)} There exists non-decreasing function $\widetilde{L}_{\widetilde{\mathfrak{W}}}\in C([a,\infty ),\mathbb{R}%
_{+})$ such that 
\begin{equation*}
|\widetilde{\mathfrak{W}}\left( x\right) (t)-\widetilde{\mathfrak{W}}\left( y\right) (t)|\leq \widetilde{L}%
_{\widetilde{\mathfrak{W}}}(t)|x(t)-y(t)|
\end{equation*}%
for $x,y\in C_{1-\xi ,\Psi }[a,\infty )$ and $t\in \lbrack a,\infty )$;

{\rm ($\tilde{d}$)} The function $\varphi \in C_{1-\xi ,\Psi }[a,\infty )$ is
increasing;

{\rm ($\tilde{e}$)} There exists $\lambda >0$ such that 
\begin{equation}
\dfrac{1}{\Gamma (\mu )}\int_{a}^{t}\mathbf{Q}_{\Psi }^{\mu }\left( t,s\right)
\varphi (s)\,{\mbox{d}}s\leq \lambda \,\varphi (t)
\end{equation}%
with $t\in \lbrack a,\infty )$.

Then,

{\rm(1)} the fractional problem {\rm Eq.(\ref{1})}--{\rm Eq.(\ref{11})}
has a unique solution in $C_{1-\xi,\Psi}([a,\infty),\mathbb{R})$;

{\rm(2)} the solution of the fractional problem {\rm Eq.(\ref{1})}--{\rm Eq.(\ref{11})} is generalized Ulam-Hyers-Rassias stable with respect to $\varphi$.
\end{teorema}

\begin{proof}
With the conditions ($\tilde{a}$), ($\tilde{b}$) and ($\tilde{c}$) and following the steps in the proof of {Theorem \ref 
{T0}},  one can easily prove that  
the problem \textrm{Eq.(\ref{1})}--{\rm Eq.(\ref{11})}, has a
unique solution in $C_{1-\xi,\Psi}^1([a,\infty),\mathbb{R})$. 

Let $y \in C_{1-\xi,\Psi}^1 ([a,\infty),\mathbb{R})$ satisfying the inequality {\rm(\ref{3})}. Let $x \in C_{1-\xi,\Psi}^1([a,\infty),\mathbb{R})$ is a unique solution of the following fractional Cauchy problem
\begin{equation*}
\left\{ 
\begin{array}{ccl}
{}^{\mathbf{H}}{\mathbf{D}}_{a^{+}}^{\mu ,\eta ,\Psi }x(t) & = & f\left( t,x(t),%
\widetilde{\mathfrak{W}}\left( x\right) (t)\right) ,~ t\in \lbrack a,\infty ) \\ 
\mathbf{I}_{a{+}}^{1-\xi ,\Psi }x(a) & = & y(a).
\end{array}
\right. 
\end{equation*}

Then, its equivalent to the Volterra integral equation is 
\begin{equation}
x(t)=\mathbf{M}^{\Psi}_{\xi}(t,a)y(a)+\frac{1}{%
\Gamma (\mu )}\int_{a}^{t}\mathbf{Q}_{\Psi }^{\mu }\left( t,s\right) f\left( s,x(s),\widetilde{\mathfrak{W}}\left( x\right) (s)\right) \,{\mbox{d}}s,  \label{B}
\end{equation}%
with $\,\,t\in \lbrack a,\infty )$.

Since $y \in C_{1-\xi,\Psi}^1 (\Delta,\mathbb{R})$ satisfies inequality {\rm (\ref{3})}, by means of the {\rm Remark \ref{R55}}, $y$ satisfies the following fractional integral inequality
\begin{eqnarray}\label{C}
&&\left\vert y(t)-\mathbf{M}^{\Psi}_{\xi}(t,a)
y(a)-\dfrac{1}{\Gamma (\mu )}\int_{a}^{t}\mathbf{Q}_{\Psi }^{\mu }\left(
t,s\right) f\left( s,x(s),\widetilde{\mathfrak{W}}\left( x\right) (s)\right) \,{\mbox{d}%
}s\right\vert   \notag \\
&\leq &\dfrac{1}{\Gamma (\mu )}\int_{a}^{t}\mathbf{Q}_{\Psi }^{\mu }\left(
t,s\right) \varphi (s)\,{\mbox{d}}s\leq \lambda \varphi (t)
\end{eqnarray}%
with $\,\,t\in \lbrack a,\infty )$.

From {\rm Eq.(\ref{B})} and {\rm Eq.(\ref{C})}, we obtain
\begin{eqnarray*}
&&\left\vert y(t)-x(t)\right\vert  \\
&=&\left\vert y(t)-\mathbf{M}^{\Psi}_{\xi}(t,a)%
y(a)-\frac{1}{\Gamma (\mu )}\int_{a}^{t}\mathbf{Q}_{\Psi }^{\mu }\left(
t,s\right) f\left( s,x(s),\widetilde{\mathfrak{W}}\left( x\right) (s)\right) \,{\mbox{d}%
}s\right\vert  \\
&\leq &\left\vert y(t)-\mathbf{M}^{\Psi}_{\xi}(t,a)y(a)-\dfrac{1}{\Gamma (\mu )}\int_{a}^{t}\mathbf{Q}_{\Psi }^{\mu }\left(
t,s\right) f\left( s,y(s),\widetilde{\mathfrak{W}}\left( y\right) (s)\right) \,{\mbox{d}%
}s\right\vert +   \\
&&+\dfrac{1}{\Gamma (\mu )}\int_{a}^{t}\mathbf{Q}_{\Psi }^{\mu }\left(
t,s\right) \left\vert f\left( s,y(s),\widetilde{\mathfrak{W}}\left( y\right) (s)\right)
-f\left( s,x(s),\widetilde{\mathfrak{W}}\left( x\right) (s)\right) \right\vert \,{%
\mbox{d}}s  \\
&\leq &\lambda \varphi (t)+\dfrac{1}{\Gamma (\mu )}\int_{a}^{t}\mathbf{Q}_{\Psi
}^{\mu }\left( t,s\right) \widetilde{L}_{f}(s)\left\{ |y(s)-x(s)|+%
\widetilde{\mathfrak{W}}\left( y\right) (s)-\widetilde{\mathfrak{W}}\left( x\right) (s)\right\} \,{\mbox{d}}s \\
&\leq &\lambda \varphi (t)+\frac{1}{\Gamma (\mu )}\int_{a}^{t}\mathbf{Q}_{\Psi
}^{\mu }\left( t,s\right) \widetilde{L}_{f}(s)(1+\widetilde{L}%
_{\widetilde{\mathfrak{W}}}(s))|y(s)-x(s)|\,{\mbox{d}}s\\
&\leq &\lambda \varphi (t)+\frac{\widetilde{L}_{f}(t)(1+\widetilde{L}%
_{\widetilde{\mathfrak{W}}}(t))}{\Gamma(\mu )}\int_{a}^{t}\mathbf{Q}_{\Psi
}^{\mu }\left( t,s\right) |y(s)-x(s)|\,{\mbox{d}}s.
\end{eqnarray*}

Using the Gronwall lemma {\rm(Lemma \ref{L1})}, we have for $t\in
\lbrack a,\infty )$ 
\begin{equation*}
|y(t)-x(t)|\leq \lambda \varphi (t)\mathcal{E}_{\mu }\left[{\widetilde{L}_{f}(t)(1+\widetilde{L}%
_{\widetilde{\mathfrak{W}}}(t))}(\Psi (t)-\Psi (a))^{\mu }\right]
 \end{equation*}
 
Therefore
\begin{equation}
(\Psi (t)-\Psi (a))^{1-\xi }| y(t) - x(t) |\leq c_{\varphi }\varphi (t), ~t\in
[a,\infty )
 \label{C1}
\end{equation}%
where $c_{\varphi }=\lambda \,\tilde{\Psi}\,\mathcal{E}_{\mu }\left( \tilde{K}\right) $,~
$\tilde{\Psi}=\underset{t\in [a,\infty )}{\sup }%
\left\vert (\Psi (t)-\Psi (a))^{1-\xi }\right\vert < \infty $ and $$\tilde{K}=\underset{t\in [a,\infty )}{\sup }%
\left\vert {\widetilde{L}_{f}(t)(1+\widetilde{L}%
_{\widetilde{\mathfrak{W}}}(t))\,(\Psi (t)-\Psi (a))^{\mu }}\right\vert < \infty.$$ 
 So, by {\rm Eq.(\ref{C1})}, the solution of the problem {\rm Eq.(\ref{1})}--{\rm Eq.(\ref{11})} is generalized Ulam-Hyers-Rassias stable with respect to $
\varphi $. 
\end{proof}

We can conclude that the investigation that was carried out by means of Theorem \ref{T0} and Theorem \ref{T6} enabled a new class of Ulam--Hyers and Ulam-Hyers-Rassias stabilities on the compact interval $\Delta=[a,b]$ and on  $I=[a,\infty )$. In this sense, the new results presented here are indeed important for the fractional calculus, in particular, for fractional analysis, enabling a growth in the area. In the following section, we present a brief application of this new class of Ulam--Hyers stability, as particular cases of Theorem \ref{T0} and Theorem \ref{T6}.

%%%%%%%%%%%%%%%%%%%%%%%%%%%%%%%%%%%%%%%%%%%%%%%%%%%%%%%%%%%%%%%%%%%%%%%%%%%%%%%%%%%%%%%%%%%%%%%%%%%%%%%%%%%%%%%%%%%%%%%%%%%%%%%%%%%%%%%%%%%%%%%%%%%%%%%%%%%%%%%%
\section{Application}

In this section as application, we will discuss a fractional Cauchy problem involving the Hadamard derivative and the stability of its solution, in both cases, Ulam--Hyers and generalized Ulam--Hyers--Rassias. The results are presented as theorems.

Consider the following fractional Cauchy problem 
\begin{equation*}
{}^{HD}\mathcal{D}_{a_{+}}^{\mu }x(t)=f\left( t,x(t),\frac{1}{\Gamma
(\mu )}\int_{a}^{t}\ln \left( \dfrac{t}{s}\right) ^{\mu -1}K(t,s,x(s))\,%
{\mbox{d}}s\right) ,\quad t\in I 
\end{equation*}
where $I= [0,1]$ or $[ 0,\infty)$, ~${}^{HD}\mathcal{D}_{a_{+}}^{\mu }(\cdot )$ is the Hadamard
fractional derivative and the inequalities 
\begin{equation*}
\left\vert {}^{HD}\mathcal{D}_{a_{+}}^{\mu }y(t)-f\left( t,y(t),\dfrac{1}{%
\Gamma (\mu )}\int_{a}^{t}\ln \left( \frac{t}{s}\right) ^{\mu
-1}K(t,s,y(s))\,{\mbox{d}}s\right) \right\vert \leq \varepsilon ,
\end{equation*}
\begin{equation*}
\left\vert {}^{HD}\mathcal{D}_{a_{+}}^{\mu }y(t)-f\left( t,y(t),\dfrac{1}{%
\Gamma (\mu )}\int_{a}^{t}\ln \left( \frac{t}{s}\right) ^{\mu
-1}K(t,s,y(s))\,{\mbox{d}}s\right) \right\vert \leq \varphi (t)
\end{equation*}
for $t\in I$.

In this case, {Theorem \ref{T0}} and {Theorem \ref{T6}} become

\begin{teorema}\label{TA1}
Suppose that

{\rm (A1)} $f \in C_{1-\xi,\Psi} ([0,1]\times \mathbb{R}^2, \mathbb{R})$; $ \widetilde{\mathfrak{W}} \in C_{1-\xi,\Psi}((C_{1-\xi,\Psi}[a,b],\mathbb{R}), 
(C_{1-\xi,\Psi}[a,b],\mathbb{R}))$; 

{\rm (A2)} There exists $L_f > 0$ such that, for all $t \in [0,1]$, 
$$
|f(t,u_1,u_2)-f(t,v_1,v_2)| \leq L_f \sum_{i=1}^2 |u_i - v_i|
$$
with $u_i,v_i \in \mathbb{R}$ and $i=1,2$;

0{\rm (A3)} There exists $L_K > 0$ such that
$$
| K(t,s,x(s)) - K(t,s,y(s))| \leq L_K |x(t) - y(t)|
$$
with $x,y \in C_{1-\xi, \Psi}[0,1]$ and $t \in [0,1]$.

{\rm (A4)} $\displaystyle \dfrac{\left(\ln\left(\dfrac{b}{a}\right)\right)^{\mu}}{\Gamma(2\mu)} \, \Gamma(\mu) \, L_f(1+L_{K}) < 1$.

Then

{\rm (i)} The problem {\rm Eq.(\ref{1})}--{\rm Eq.(\ref{11})} has a unique solution in $C_{1-\xi,\Psi}([0,1],\mathbb{R})$.

{\rm (ii)} The solution of {\rm Eq.(\ref{1})} is Ulam--Hyers stable.
\end{teorema}

\begin{proof}
The proof follows the same steps as in {Theorem \ref{T0}}. 
\end{proof}

\begin{teorema}\label{TA2}
We assume that

{\rm (B1)} $f \in C_{1-\xi,\Psi} ([0,\infty)\times \mathbb{R}^2, \mathbb{R})$; $ \widetilde{\mathfrak{W}} \in C((C_{1-\xi,\Psi}[0,1],\mathbb{R}), (C_{1-\xi,\Psi}[0,1],\mathbb{R}))$; 

{\rm (B2)} There exists non-decreasing function $\widetilde{L}_f \in C ([0,\infty),\mathbb{R}_{+}) $ such that, for all $t \in [0,\infty]$, 
$$
|f(t,u_1,u_2)-f(t,v_1,v_2)| \leq \widetilde{L}_f(t) (|u_1 - v_1| + |u_2-v_2|)
$$
with $u_i,v_i \in \mathbb{R}$ and $i=1,2$;

{\rm (B3)} There exists non-decreasing function $\widetilde{L}_K \in C([0,\infty),\mathbb{R})$ such that
$$
| K(t,s,u)) - K(t,s,v)| \leq \widetilde{L}_K (t) |x(t) - y(t)|
$$
with $x,y \in C_{1-\xi, \Psi}[0,\infty)$ and $t \in [0,\infty)$.

{\rm (B4)} The function $\varphi \in C_{1-\xi,\Psi}[0,\infty)$ is increasing;

{\rm (B5)} There exists $\lambda > 0$ such that
$$
\frac{1}{\Gamma(\mu)} \int_a^t \ln\left( \frac{t}{s} \right)^{\mu -1} \varphi(s) \, {\mbox{d}}s \leq \lambda \varphi(t), \quad t \in [0,\infty).
$$

Then

{\rm (i)} The problem {\rm Eq.(\ref{1})}--{\rm Eq.(\ref{11})} has a unique solution in $C_{1-\xi,\Psi}([0,\infty),\mathbb{R})$.

{\rm (ii)} The solution of {\rm Eq.(\ref{1})} is Ulam-Hyers-Rassias stable with respect to $\varphi$.
\end{teorema}

\begin{proof}
The proof follows the same steps as in { Theorem \ref{T6}}. 
\end{proof}

\section{Concluding Remarks}

The investigation of Ulam--Hyers stabilities of solutions of several types of fractional differential equations is a major motivation for researchers. Here we investigated a new class of Ulam--Hyers stabilities of the differential equation with abstract Volterra operator introduced by means of the $\Psi$-Hilfer fractional derivative on the compact interval $\Delta=[a,b]$ and on the infinite interval $I=[a,\infty)$, making use of Banach's fixed point theorem and Gronwall inequality. In this paper, the results were obtained in the Banach space. It will be whether there is a possibility of investigating such results in other spaces, such as Fréchet and fractional Sobolev spaces. Researches in this follow-up are being worked on and future works will be published.

\section*{Acknowledgment}
{\bf I support financial support of the PNDP-CAPES scholarship of the Pos-Graduate Program in Mathematics UFPA/UFAM.}

\bibliographystyle{amsplain}

\end{document}